\documentclass[oneside,american,amsaddr]{amsart}
\usepackage[T1]{fontenc}
\usepackage[latin9]{inputenc}
\usepackage{verbatim}
\usepackage{units}
\usepackage{amstext}
\usepackage{amsthm}
\usepackage{amssymb}

\makeatletter
\numberwithin{equation}{section}
\numberwithin{figure}{section}
  \theoremstyle{remark}
  \newtheorem*{rem*}{\protect\remarkname}
\theoremstyle{plain}
\newtheorem{thm}{\protect\theoremname}[section]
  \theoremstyle{plain}
  \newtheorem{lem}[thm]{\protect\lemmaname}
  \theoremstyle{plain}
  \newtheorem*{conjecture*}{\protect\conjecturename}
  \theoremstyle{plain}
  \newtheorem{prop}[thm]{\protect\propositionname}
  \theoremstyle{plain}
  \newtheorem{cor}[thm]{\protect\corollaryname}

\makeatother

\usepackage{babel}
  \providecommand{\conjecturename}{Conjecture}
  \providecommand{\corollaryname}{Corollary}
  \providecommand{\lemmaname}{Lemma}
  \providecommand{\propositionname}{Proposition}
  \providecommand{\remarkname}{Remark}
\providecommand{\theoremname}{Theorem}

\begin{document}

\title{Competition in periodic media: I \textendash{} Existence of pulsating
fronts}

\author{Léo Girardin$^{1}$}

\thanks{$^{1}$ Laboratoire Jacques-Louis Lions, CNRS UMR 7598, Université
Pierre et Marie Curie, 4 place Jussieu, 75005 Paris, France}

\email{$^{1}$ girardin@ljll.math.upmc.fr}
\begin{abstract}
This paper is concerned with the existence of pulsating front solutions
in space-periodic media for a bistable two-species competition\textendash diffusion
Lotka\textendash Volterra system. Considering highly competitive systems,
a simple \textquotedblleft high frequency or small amplitudes\textquotedblright{}
algebraic sufficient condition for the existence of pulsating fronts
is stated. This condition is in fact sufficient to guarantee that
all periodic coexistence states vanish and become unstable as the
competition becomes large enough.
\end{abstract}



\maketitle
\tableofcontents{}

\section*{Introduction}

This is the first part of a sequel to our previous article with Grégoire
Nadin \cite{Girardin_Nadin_2015}. In this prequel, we studied the
sign of the speed of bistable traveling wave solutions of the following
competition\textendash diffusion problem:
\[
\left\{ \begin{matrix}\partial_{t}u_{1}-\partial_{xx}u_{1}=u_{1}\left(1-u_{1}\right)-ku_{1}u_{2} & \mbox{ in }\left(0,+\infty\right)\times\mathbb{R}\,\\
\partial_{t}u_{2}-d\partial_{xx}u_{2}=ru_{2}\left(1-u_{2}\right)-\alpha ku_{1}u_{2} & \mbox{ in }\left(0,+\infty\right)\times\mathbb{R}.
\end{matrix}\right.
\]

We proved that, as $k\to+\infty$, the speed of the traveling wave
connecting $\left(1,0\right)$ to $\left(0,1\right)$ converges to
a limit which has exactly the sign of $\alpha^{2}r-d$. In particular,
if $\alpha=r=1$ and if $k$ is large enough, the more motile species
is the invader: this is what we called the \textquotedblleft unity
is not strength\textquotedblright{} result.

In view of this result, it would seem natural to try to generalize
it in heterogeneous spaces, that is to systems with non-constant coefficients.
Is the more motile species still the invading one?

The first obstacle toward this generalization is that of the existence
of traveling fronts \textendash or of some suitable generalization
of these\textendash{} for such a problem. Indeed, while past work
had already established the existence of competitive bistable traveling
waves in the case of homogeneous spaces (recall for instance Gardner
\cite{Gardner_1982} and Kan-On \cite{Kan_on_1995}), to the best
of our knowledge, there is at this time no such pre-established result
in the case of fully heterogeneous spaces (see the recent review of
Guo and Wu \cite{Guo_Wu_2012}).

One of the main difficulties regarding this existence problem is of
course the combination of unboundedness and heterogeneity. This yields
additional difficulties (for instance, there are multiple non-equivalent
definitions of the principal eigenvalue \cite{Berestycki_Ros} and
convenient integration-wise boundary conditions are lacking). Therefore,
it is likely easier to first treat a simple case. With this in mind,
we focus in this article on a simple, yet relevant application-wise
heterogeneity: the periodic one. We hope to pave the way for a possible
future generalization.

Periodic spaces are likely the type of unbounded heterogeneous spaces
we know best how to handle mathematically and thus a literature about
scalar equations in periodic spaces has been developed during the
past few years. Concerning scalar reaction\textendash diffusion in
periodic spaces and with ``KPP\textquotedblright -type non-linearities,
important results have been established recently by Berestycki and
his collaborators \cite{Berestycki_Ham_3,Berestycki_Ham_1,Berestycki_Ham_2}
(see also Nadin \cite{Nadin_2009,Nadin_2011} in space-time periodic
media). We will rely a lot on these scalar results. Regarding bistable
non-linearities, we refer to the work of Ding, Hamel and Zhao \cite{Ding_Hamel_Zhao}
and Zlatos \cite{Zlatos_2015}.

For the sake of simplicity, we will assume that diffusion and interspecific
competition rates are constant. We expect our main ideas to be generalizable
to systems with periodic diffusion and interspecific competition rates,
but we also expect a lot of technical details to get messy and there
might very well be some major issues. As a counterpart to this loss
in generality, we will be able to treat a much larger class of growth\textendash saturation
terms since the explicit form of these will not be prescribed a priori.
We will only require some reasonable \textquotedblleft KPP non-linearities\textquotedblright{}
assumptions. 

Since our final goal is to study the limits of these pulsating fronts
as the competition becomes infinite, we will only consider systems
in which competition is the main underlying mechanism, that is for
large values of the interspecific competition rate. A first consequence
of this approach is that our system will always be bistable. A second
consequence is that segregation phenomena will be involved quite frequently.
Competition-induced segregation in homogeneous spaces have been a
main center of interest of Dancer, Terracini and others since the
nineties (\cite{Conti_Terracin,Crooks_Dancer_,Crooks_Dancer__1,Dancer_Du_1994,Dancer_Guo_199,Dancer_Hilhors,Dancer_2010,Dancer06}
among others). They basically confirmed the intuitive idea that competitors
tend to live in different ecological niches. 

To investigate the existence of bistable pulsating fronts connecting
two extinction states, we have at our disposal recent abstract results
about monotone semiflows stated by Weinberger \cite{Weinberger_200}
(monostable case) and Fang and Zhao \cite{Fang_Zhao_2011} (bistable
case). Even though both articles were mostly concerned by scalar equations,
they were careful enough to include monotone systems, such as two-species
competitive ones, in their framework. Notice that Yu and Zhao \cite{Yu_Zhao_2015}
used a similar framework to prove, in the weak competition case, the
existence of monostable pulsating fronts connecting two extinction
states despite the presence of an intermediate coextinction state
(Weinberger\textquoteright s framework requires no intermediate stationary
state) (see also Fang\textendash Yu\textendash Zhao \cite{Fang_Yu_Zhao}
for a similar work in space-time periodic media). 

The core idea of Fang and Zhao\textquoteright s theorem is as follows:
provided a bistable monotone problem, if all intermediate stationary
states are unstable and if they are invaded by the stable states,
then bistable traveling waves do exist. While these hypotheses might
be easily verified for some problems (say, scalar or space-homogeneous),
in the case exposed here, real issues arise from the segregation phenomenon.
Indeed, stable intermediate segregated periodic coexistence states
might a priori exist. Therefore it is natural to wonder whether periodicity
might induce some simple, yet relevant, sufficient condition to enforce
the non-existence of segregated periodic coexistence states. We will
indeed state one such condition and will show that this condition
is moreover sufficient to guarantee that all remaining periodic stationary
states are unstable and invaded by the stable ones. 

The following pages will be organized as follows: in the first section,
the core hypotheses and framework will be precisely formulated and
the main results stated. The second section will be dedicated to the
proof of the existence of pulsating front solutions; in particular,
we will perform a quite thorough study of the stability of periodic
coexistence states. 

The study of the limit as $k\to+\infty$ of these pulsating fronts
will be the object of the second part \cite{Girardin_Nadin_2016}.

\section{Preliminaries and main results}

Let $d,k,\alpha,L>0$, $C=\left(0,L\right)\subset\mathbb{R}$ and
$\left(f_{1},f_{2}\right):[0,+\infty)\times\mathbb{R}\to\mathbb{R}^{2}$
$L$-periodic with respect to its second variable. For any $u:\mathbb{R}^{2}\to[0,+\infty)$
and $i\in\left\{ 1,2\right\} $, we refer to $\left(t,x\right)\mapsto f_{i}\left(u\left(t,x\right),x\right)$
as $f_{i}\left[u\right]$. Our interest lies in the following competition\textendash diffusion
problem:
\[
\left\{ \begin{matrix}\partial_{t}u_{1}=\partial_{xx}u_{1}+u_{1}f_{1}\left[u_{1}\right]-ku_{1}u_{2}\\
\partial_{t}u_{2}=d\partial_{xx}u_{2}+u_{2}f_{2}\left[u_{2}\right]-\alpha ku_{1}u_{2}
\end{matrix}\right.\quad\left(\mathcal{P}_{k}\right)
\]

\subsection{Preliminaries}

\subsubsection{Redaction conventions. }
\begin{itemize}
\item Mirroring the definition of $f_{1}\left[u\right]$ and $f_{2}\left[u\right]$,
for any function of two real variables $f$ and any real-valued function
$u$ of two real variables, $f\left[u\right]$ will refer to $\left(t,x\right)\mapsto f\left(u\left(t,x\right),x\right)$.
For any real-valued function $u$ of one real variable, $f\left[u\right]$
will refer to $x\mapsto f\left(u\left(x\right),x\right)$. For any
function $f$ of one real variable and any real-valued function $u$
of one or two real variables, $f\left[u\right]$ will simply refer
to $f\circ u$.
\item For the sake of brevity, although we could index everything ($\left(\mathcal{P}\right)$,
$u_{1}$, $u_{2}$\dots ) on $k$ and $d$, the dependencies on $k$
or $d$ will mostly be implicit and will only be made explicit when
it definitely facilitates the reading.
\item Since we consider the limit of this system when $k\to+\infty$, many
(but finitely many) results will only be true when \textquotedblleft $k$
is large enough\textquotedblright . Hence, we define by induction
the positive number $k^{\star}$, whose value is initially $1$ and
is updated each time a statement is only true when \textquotedblleft $k$
is large enough\textquotedblright{} in the following way: if the statement
is true for any $k\geq k^{\star}$, the value of $k^{\star}$ is unchanged;
if, conversely, there exists $K>k^{\star}$ such that the statement
is true for any $k\geq K$ but false for any $k\in[k^{\star},K)$,
the value of $k^{\star}$ becomes that of $K$. In the text, we will
indifferently write \textquotedblleft for $k$ large enough\textquotedblright{}
or \textquotedblleft provided $k^{\star}$ is large enough\textquotedblright .
Moreover, when $k$ indexes appear, they a priori indicate that we
are considering families indexed on (equivalently, functions defined
on) $[k^{\star},+\infty)$, but for the sake of brevity, when sequential
arguments imply extractions of sequences and subsequences indexed
themselves on increasing elements of $[k^{\star},+\infty)^{\mathbb{N}}$,
we will not explicitly define these sequences of indexes and will
simply stick with the indexes $k$, reindexing along the course of
the proof the considered objects. In such a situation, the statement
``as $k\to+\infty$'' should be understood unambiguously.
\item Periodicity will always implicitly mean $L$-periodicity (unless explicitly
stated otherwise). For any functional space $X$ on $\mathbb{R}$,
$X_{per}$ denotes the subset of $L$-periodic elements of $X$.
\item We will use the classical partial order on the space of functions
from any $\Omega\subset\mathbb{R}^{N}$ to $\mathbb{R}$: $g\leq h$
if and only if, for any $x\in\Omega$, $g\left(x\right)\leq h\left(x\right)$,
and $g<h$ if and only if $g\leq h$ and $g\neq h$. We recall that
when $g<h$, there might still exists $x\in\Omega$ such that $g\left(x\right)=h\left(x\right)$.
If, for any $x\in\Omega$, $g\left(x\right)<h\left(x\right)$, we
use the notation $g\ll h$. In particular, if $g\geq0$, we say that
$g$ is non-negative, if $g>0$, we say that $g$ is non-negative
non-zero, and if $g\gg0$, we say that $g$ is positive (and we define
similarly non-positive, non-positive non-zero and negative functions).
Eventually, if $g_{1}\leq h\leq g_{2}$, we write $h\in\left[g_{1},g_{2}\right]$,
if $g_{1}<h<g_{2}$, we write $h\in\left(g_{1},g_{2}\right)$, and
if $g_{1}\ll h\ll g_{2}$, we write $h\in\left\langle g_{1},g_{2}\right\rangle $. 
\item We will also use the partial order on the space of vector functions
$\Omega\to\mathbb{R}^{N'}$ naturally derived from the preceding partial
order. It will involve similar notations.
\item The periodic principal eigenvalue of a second order elliptic operator
$\mathcal{L}$ with periodic coefficients will be generically referred
to as $\lambda_{1,per}\left(-\mathcal{L}\right)$. Recall (from Berestycki\textendash Hamel\textendash Roques
\cite{Berestycki_Ham_1} for instance) that the periodic principal
eigenvalue of $\mathcal{L}$ is the unique real number $\lambda$
such that there exists a periodic function $\varphi\gg0$ satisfying:
\[
\left\{ \begin{matrix}-\mathcal{L}\varphi=\lambda\varphi\mbox{ in }\mathbb{R}\\
\|\varphi\|_{L^{\infty}\left(C\right)}=1
\end{matrix}\right.
\]
The Dirichlet principal eigenvalue of an elliptic operator $\mathcal{L}$
in a sufficiently smooth domain $\Omega$ will be referred to as $\lambda_{1,Dir}\left(-\mathcal{L},\Omega\right)$.
Since our framework is spatially one-dimensional, such elliptic operators
will involve first and second derivatives with respect to the spatial
variable $x$.
\end{itemize}

\subsubsection{Hypotheses on the reaction.}

For any $i\in\left\{ 1,2\right\} $, we have in mind functions $f_{i}$
such that the reaction term $uf_{i}\left[u\right]$ is of logistic
type (also known as KPP type). At least, we want to cover the largest
possible class of $\left(u,x\right)\mapsto\mu\left(x\right)-\nu\left(x\right)u$.
This is made precise by the following assumptions.

{ \renewcommand\labelenumi{($\mathcal{H}_\theenumi$)}
\begin{enumerate}
\item $f_{i}$ is $\mathcal{C}^{1}$ with respect to its first variable
up to $0$ and Hölder-continuous with respect to its second variable
with a Hölder exponent larger than or equal to $\frac{1}{2}$.
\item There exists a constant $m_{i}>0$ such that $f_{i}\left[0\right]\geq m_{i}$.
\item $f_{i}$ is decreasing with respect to its first variable and there
exists $a_{i}>0$ such that, if $u>a_{i}$, then for any $x\in\mathbb{R}$
$f_{i}\left(u,x\right)<0$.
\end{enumerate}
}
\begin{rem*}
If $f_{i}$ is in the class of all $\left(u,x\right)\mapsto\mu\left(x\right)-\nu\left(x\right)u$,
then $\mu,\nu\in\mathcal{C}_{per}^{0,\nicefrac{1}{2}}\left(\mathbb{R}\right)$,
$\mu\gg0$, $\nu\gg0$. More generally, from $\left(\mathcal{H}_{1}\right)$,
$\left(\mathcal{H}_{2}\right)$ and the periodicity of $f_{i}\left[0\right]$,
it follows immediately that there exists a constant $M_{i}>m_{i}$
such that $f_{i}\left[0\right]\leq M_{i}$. Without loss of generality,
we assume that $m_{i}$ and $M_{i}$ are optimal, that is $m_{i}=\min\limits _{\overline{C}}f_{i}\left[0\right]$
and $M_{i}=\max\limits _{\overline{C}}f_{i}\left[0\right]$. 
\end{rem*}
We refer to $\max\left(M_{1},M_{2}\right)$ (resp. $\min\left(m_{1},m_{2}\right)$)
as $M$ (resp. $m$).

Furthermore, we need a coupled hypothesis on the pair $\left(f_{1},f_{2}\right)$. 

{ \renewcommand\labelenumi{($\mathcal{H}_{freq}$)}
\begin{enumerate}
\item The constants $d$, $M_{1}$ and $M_{2}$ satisfy $L<\pi\left(\frac{1}{\sqrt{M_{1}}}+\sqrt{\frac{d}{M_{2}}}\right)$.
\end{enumerate}
}
\begin{rem*}
Even if this might not be clear right now, this is the key hypothesis.
$\left(\mathcal{H}_{freq}\right)$ means that, given a fixed amplitude,
we consider high frequencies, or equivalently, given a fixed frequency,
we consider low amplitudes. This sufficient condition for existence
might be a bit relaxed but the best condition we can give is very
verbose and only slightly better. See the proof of Proposition \ref{prop:high_frequency_consequence},
which is where $\left(\mathcal{H}_{freq}\right)$ plays its role. 
\end{rem*}

\subsection{Two main results and a conjecture}

Using known results about scalar equations and periodic principal
eigenvalues \cite{Berestycki_Ham_1}, the following lemma is quite
straightforward (as will show Subsection \ref{subsec:Extinction-states}).
\begin{lem}
Assume that $f_{1}$ and $f_{2}$ satisfy $\left(\mathcal{H}_{1}\right)$,
$\left(\mathcal{H}_{2}\right)$ and $\left(\mathcal{H}_{3}\right)$.

The set of all periodic stationary states of the problem $\left(\mathcal{P}\right)$
contains $\left(0,0\right)$, which is unstable, and a pair $\left\{ \left(\tilde{u}_{1},0\right),\left(0,\tilde{u}_{2}\right)\right\} $
with $\left(\tilde{u}_{1},\tilde{u}_{2}\right)\in\mathcal{C}_{per}^{2}\left(\mathbb{R},\left(0,+\infty\right)^{2}\right)$. 
\end{lem}
As usual in the literature concerning competitive systems, hereafter,
the stationary states with exactly one null component are referred
to as \textit{extinction states} whereas the stationary states with
no null component are referred to as \textit{coexistence states}.
The extinction states of $\left(\mathcal{P}\right)$ are periodic
and some of its coexistence states may be periodic as well.

Our contribution to the study of the stationary states is the following
theorem.
\begin{thm}
Assume that $f_{1}$ and $f_{2}$ satisfy $\left(\mathcal{H}_{1}\right)$,
$\left(\mathcal{H}_{2}\right)$ and $\left(\mathcal{H}_{3}\right)$
and that $\left(f_{1},f_{2}\right)$ satisfies $\left(\mathcal{H}_{freq}\right)$.

Then there exists $k^{\star}>0$ such that, for any $k>k^{\star}$,
each extinction state is locally asymptotically stable and any periodic
coexistence state is unstable. 

Furthermore, let $\left(u_{1,k},u_{2,k}\right)_{k>k^{\star}}$ be
a family of $\mathcal{C}_{per}^{2}\left(\mathbb{R},\mathbb{R}^{2}\right)$
such that, for any $k>k^{\star}$, $\left(u_{1,k},u_{2,k}\right)$
is an unstable periodic stationary state of $\left(\mathcal{P}_{k}\right)$.
Then $\left(u_{1,k},u_{2,k}\right)$ converges in $\mathcal{C}_{per}\left(\mathbb{R},\mathbb{R}^{2}\right)$
to $\left(0,0\right)$ as $k\to+\infty$.
\end{thm}
\begin{rem*}
We stress that we did not investigate the existence nor the countability
of the subset of periodic coexistence states. We stress as well that
we did not investigate at all aperiodic coexistence states. We believe
that a sharper description of the set of stationary states of $\left(\mathcal{P}\right)$
could follow from bifurcation arguments (see Hutson\textendash Lou\textendash Mischaikow
\cite{Hutson02} or Furter\textendash López-Gómez \cite{Furter_Lopez_G}).
Since it was not our point at all (instability of periodic coexistence
states was only a required step toward existence of pulsating fronts),
we chose to leave this subject as an open question. 
\end{rem*}
Thanks to the previous theorem, it is then possible to prove the following
existence theorem.
\begin{thm}
Assume that $f_{1}$ and $f_{2}$ satisfy $\left(\mathcal{H}_{1}\right)$,
$\left(\mathcal{H}_{2}\right)$ and $\left(\mathcal{H}_{3}\right)$
and that $\left(f_{1},f_{2}\right)$ satisfies $\left(\mathcal{H}_{freq}\right)$.

Then there exists $k^{\star}>0$ such that, for any $k>k^{\star}$,
the problem $\left(\mathcal{P}\right)$ admits a bistable pulsating
front solution connecting the two extinction states. 
\end{thm}
To end this subsection, let us present an important conjecture about
the existence problem and about the sharpness of $\left(\mathcal{H}_{freq}\right)$.
We did not address this question but hopefully others will.
\begin{conjecture*}
\label{conjecture_existence} Neither $\left(\mathcal{H}_{freq}\right)$
nor the nonexistence of a stable periodic coexistence state are necessary
conditions for the existence of a bistable pulsating front solution
connecting the two extinction states. 

Furthermore, there exists a non-empty set of parameters $\left(L,d,\alpha,k,f_{1},f_{2}\right)$
such that no such pulsating front exists.
\end{conjecture*}
We point out that, according to the present work, any of the following
two conditions enforces that either $\left(\mathcal{H}_{freq}\right)$
is not satisfied or $k\leq k^{\star}$: 
\begin{itemize}
\item the existence of a stable periodic coexistence state;
\item the nonexistence of a bistable pulsating front solution.
\end{itemize}
Moreover, our work will show that, if $k>k^{\star}$, any stable periodic
coexistence state has the \textquotedblleft close to segregation\textquotedblright{}
form (which will be rigorously defined later on; roughly speaking,
\textquotedblleft close to segregation\textquotedblright{} periodic
coexistence states converge as $k\to+\infty$ to a non-trivial periodic
coexistence state satisfying $u_{1}u_{2}=0$). This important property
might be the starting point of a future work on the preceding conjecture.

\subsection{A few more preliminaries}

\subsubsection{Compact embeddings of Hölder spaces}

We recall a well-known result of functional analysis.
\begin{prop}
Let $\left(a,a'\right)\in\left(0,+\infty\right)^{2}$ and $n,n',\beta,\beta'$
such that $\left(a,a'\right)=\left(n+\beta,n'+\beta'\right)$, $n$
and $n'$ are non-negative integers and $\beta$ and $\beta'$ are
in $(0,1]$.

If $a\leq a'$, then the canonical embedding $i:\mathcal{C}^{n',\beta'}\left(C\right)\hookrightarrow\mathcal{C}^{n,\beta}\left(C\right)$
is continuous and compact. 
\end{prop}
It will be clear later on that this problem naturally involves uniform
bounds in $\mathcal{C}^{0,\nicefrac{1}{2}}$ and in $\mathcal{C}^{2,\nicefrac{1}{2}}$.
Therefore, we fix once and for all $\beta\in\left(0,\frac{1}{2}\right)$
and we will use systematically the compact embeddings $\mathcal{C}^{n,\nicefrac{1}{2}}\hookrightarrow\mathcal{C}^{n,\beta}$,
meaning that uniform bounds in $\mathcal{C}^{n,\nicefrac{1}{2}}$
yield relative compactness in $\mathcal{C}^{n,\beta}$.

\subsubsection{Existence and uniqueness for the evolution system}
\begin{prop}
Let $k>0.$ Equipped with an initial non-negative condition $\left(u_{1,0},u_{2,0}\right)\in\mathcal{C}^{0,\nicefrac{1}{2}}\left(\mathbb{R},\mathbb{R}^{2}\right)$,
the problem $\left(\mathcal{P}\right)$ is well-posed: there exists
a unique non-negative entire solution $\left(u_{1},u_{2}\right)\in\mathcal{C}^{1,\nicefrac{1}{4}}\left([0,+\infty),\mathcal{C}^{2,\nicefrac{1}{2}}\left(\mathbb{R},\mathbb{R}^{2}\right)\right)$. 

Furthermore, if $\left(u_{1,0},u_{2,0}\right)>0$, then $\left(u_{1},u_{2}\right)\gg0$,
and if $\left(u_{1,0},u_{2,0}\right)\in\mathcal{C}_{per}\left(\mathbb{R},\mathbb{R}^{2}\right)$,
then $\left(u_{1},u_{2}\right)\in\mathcal{C}^{1}\left([0,+\infty),\mathcal{C}_{per}^{2}\left(\mathbb{R},\mathbb{R}^{2}\right)\right)$.
\end{prop}
\begin{rem*}
We do not give a fully detailed proof of this statement. Ideas similar
to those given in Berestycki\textendash Hamel\textendash Roques \cite[Remark 2.7]{Berestycki_Ham_1}
suffice. The existence of solutions for the truncated system in $\left(-n,n\right)$
with Dirichlet boundary conditions can be proved with Pao\textquoteright s
super- and sub-solutions theorem for competitive systems \cite{Pao_1981}. 
\end{rem*}

\subsubsection{Extinction states\label{subsec:Extinction-states}}
\begin{lem}
The periodic principal eigenvalues of $-\frac{\mbox{d}^{2}}{\mbox{d}x^{2}}-f_{1}\left[0\right]$
and $-d\frac{\mbox{d}^{2}}{\mbox{d}x^{2}}-f_{2}\left[0\right]$ are
negative.
\end{lem}
\begin{proof}
This follows from $\left(\mathcal{H}_{2}\right)$ and the monotonicity
of the periodic principal eigenvalue with respect to the zeroth order
term of the elliptic operator. Indeed, for instance: 
\[
\lambda_{1,per}\left(-\frac{\mbox{d}^{2}}{\mbox{d}x^{2}}-f_{1}\left[0\right]\right)\leq\lambda_{1,per}\left(-\frac{\mbox{d}^{2}}{\mbox{d}x^{2}}-m_{1}\right)=-m_{1}<0.
\]
\end{proof}
From this lemma and hypotheses $\left(\mathcal{H}_{1}\right)$ and
$\left(\mathcal{H}_{3}\right)$, a fundamental result from Berestycki\textendash Hamel\textendash Roques
\cite{Berestycki_Ham_1} can be applied.
\begin{thm}
For any $\delta>0$ and any $i\in\left\{ 1,2\right\} $, the equation:
\[
-\delta z''=zf_{i}\left[z\right]
\]
admits a unique positive solution in $\mathcal{C}_{per}^{2}\left(\mathbb{R}\right)$.
\end{thm}
Hereafter, $\tilde{u}_{1}$ and $\tilde{u}_{2}$ are the respective
unique positive periodic solutions of: 
\[
-z''=zf_{1}\left[z\right],
\]
\[
-dz''=zf_{2}\left[z\right].
\]

$\left(\tilde{u}_{1},0\right)$ and $\left(0,\tilde{u}_{2}\right)$
are indeed the extinction states of any $\left(\mathcal{P}_{k}\right)$. 

\subsubsection{Monotone evolution system}

One of the most important specificities of two-species competitive
systems is that, up to a slight transformation, they are monotone
systems. It is the key behind the results of Fang\textendash Zhao
\cite{Fang_Zhao_2011} and Weinberger \cite{Weinberger_200}. Let
us recall this transformation.
\begin{lem}
Let $J:z\mapsto\tilde{u}_{2}-z$, for any $z\in\mathcal{C}_{per}^{2}\left(\mathbb{R}\right)$
or $z\in\mathcal{C}^{1}\left([0,+\infty),\mathcal{C}_{per}^{2}\left(\mathbb{R}\right)\right)$
(with a slight abuse of notation). Let $k>k^{\star}$ and let $\left(u_{1},u_{2}\right)$
be a solution of $\left(\mathcal{P}\right)$ and $v_{2}=J\left(u_{2}\right)$. 

Then $\left(u_{1},v_{2}\right)$ satisfies the following cooperative
problem with periodicity conditions:
\[
\left\{ \begin{matrix}\partial_{t}u_{1}-\partial_{xx}u_{1}=u_{1}f_{1}\left[u_{1}\right]+ku_{1}\left(-\tilde{u}_{2}+v_{2}\right)\\
\partial_{t}v_{2}-d\partial_{xx}v_{2}=\tilde{u}_{2}f_{2}\left[\tilde{u}_{2}\right]-\left(\tilde{u}_{2}-v_{2}\right)f_{2}\left[\tilde{u}_{2}-v_{2}\right]+\alpha ku_{1}\left(\tilde{u}_{2}-v_{2}\right).
\end{matrix}\right.\quad\left(\mathcal{M}_{k}\right)
\]
\end{lem}
\begin{cor}
Any solution $\left(u_{1},u_{2}\right)$ of $\mathcal{\left(P\right)}$
with initial condition $\left(0,0\right)<\left(u_{1,0},u_{2,0}\right)<\left(\tilde{u}_{1},\tilde{u}_{2}\right)$
satisfies $\left(0,0\right)\ll\left(u_{1},u_{2}\right)\ll\left(\tilde{u}_{1},\tilde{u}_{2}\right)$.
\end{cor}

\subsubsection{Segregated reaction terms}

As $k\to+\infty$, the following functions will naturally appear:

\[
\eta:\left(z,x\right)\mapsto f_{1}\left(\frac{z}{\alpha},x\right)z^{+}-\frac{1}{d}f_{2}\left(-\frac{z}{d},x\right)z^{-},
\]
\[
\gamma:\left(z,x\right)\mapsto f_{1}\left(0,x\right)z^{+}-\frac{1}{d}f_{2}\left(0,x\right)z^{-},
\]
 where $z^{+}=\max\left(z,0\right)$ and $z^{-}=-\min\left(z,0\right)$
so that $z=z^{+}-z^{-}$.

\subsubsection{Derivatives of the reaction terms}

We will denote $g_{i}$ the partial derivative of $\left(u,x\right)\mapsto uf_{i}\left(u,x\right)$
with respect to $u$:
\[
g_{i}:\left(u,x\right)\mapsto f_{i}\left(u,x\right)+u\partial_{1}f_{i}\left(u,x\right)\text{ for all }i\in\left\{ 1,2\right\} .
\]

\section{Existence of pulsating fronts}

\subsection{Aim: Fang\textendash Zhao\textquoteright s theorem}

We recall that, for any $k>k^{\star}$ and any $t>0$, the Poincaré\textquoteright s
map $Q_{t}$ associated with $\left(\mathcal{M}\right)$ is defined
as the operator: 
\[
Q_{t}:\mathcal{C}\left(\mathbb{R},\mathbb{R}^{2}\right)\cap\left[\left(0,0\right),\left(\tilde{u}_{1},\tilde{u}_{2}\right)\right]\to\mathcal{C}\left(\mathbb{R},\mathbb{R}^{2}\right)\cap\left[\left(0,0\right),\left(\tilde{u}_{1},\tilde{u}_{2}\right)\right]
\]
 which associates with some initial condition $\left(u_{1,0},v_{2,0}\right)$
the solution $\left(u_{1},v_{2}\right)$ of $\left(\mathcal{M}\right)$evaluated
at time $t>0$.

From Fang and Zhao \cite{Fang_Zhao_2011}, we know that $\left(\mathcal{M}\right)$
admits a pulsating front solution connecting $\left(\tilde{u}_{1},\tilde{u}_{2}\right)$
to $\left(0,0\right)$ if:
\begin{enumerate}
\item $\left(0,0\right)$ and $\left(\tilde{u}_{1},\tilde{u}_{2}\right)\gg\left(0,0\right)$
are locally asymptotically stable periodic stationary states of $\left(\mathcal{M}\right)$
and all intermediate periodic stationary states of $\left(\mathcal{M}\right)$
are unstable;
\item for any intermediate periodic stationary state $\left(u_{1},v_{2}\right)$,
the sum of the spreading speeds associated with front-like initial
data connecting respectively $\left(\tilde{u}_{1},\tilde{u}_{2}\right)$
to $\left(u_{1},v_{2}\right)$ and $\left(u_{1},v_{2}\right)$ to
$\left(0,0\right)$ is positive (notice that these sub-problems are
of monostable type);
\item and if, for any $t>0$, $Q_{t}$ satisfies the following hypotheses:
\begin{enumerate}
\item $Q_{t}$ is spatially periodic;
\item $Q_{t}$ is continuous with respect to the topology of the locally
uniform convergence;
\item $Q_{t}$ is strongly monotone, in the sense that if $\left(u_{1},v_{2}\right)>\left(u^{1},v^{2}\right)$,
then: 
\[
Q_{t}\left(\left(u_{1},v_{2}\right)\right)\gg Q_{t}\left(\left(u^{1},v^{2}\right)\right);
\]
\item $Q_{t}$ is compact with respect to the topology of the locally uniform
convergence;
\end{enumerate}
\end{enumerate}
It is quite standard to check that the last four hypotheses are indeed
satisfied. The verification of the first two, on the contrary, is
the object of the remaining of this paper.

\subsection{Stability of all extinction states}
\begin{prop}
Provided $k^{\star}$ is large enough, $\left(\tilde{u}_{1},0\right)$
and $\left(0,\tilde{u}_{2}\right)$ are locally asymptotically stable. 
\end{prop}
\begin{rem*}
For the case $k=1$, the proof of the local asymptotic stability of
the extinction states was done by Dockery and his coauthors \cite{Dockery_1998}
with the help of Mora\textquoteright s theorem \cite{Mora_1983}.
It works here too with a very slight adaptation; we give the proof
for the sake of completeness.
\end{rem*}
\begin{proof}
Thanks to Mora\textquoteright s theorem \cite{Mora_1983}, we know
that $\left(\tilde{u}_{1},0\right)$ is asymptotically stable if the
periodic principal eigenvalue of the elliptic part of the monotone
problem $\left(\mathcal{M}\right)$ linearized at $\left(\tilde{u}_{1},\tilde{u}_{2}\right)=\left(u,J\left(0\right)\right)$
is positive. Therefore we consider the differential operator $\mathcal{A}_{\left(\tilde{u}_{1},0\right)}:\mathcal{C}_{per}^{2}\left(\mathbb{R}\right)\to\mathcal{C}_{per}\left(\mathbb{R}\right)$
defined as:
\[
\mathcal{A}_{\left(\tilde{u}_{1},0\right)}=\left(\begin{matrix}\frac{\mbox{d}^{2}}{\mbox{d}x^{2}}+g_{1}\left[\tilde{u}_{1}\right] & k\tilde{u}_{1}\\
0 & d\frac{\mbox{d}^{2}}{\mbox{d}x^{2}}+f_{2}\left[0\right]-\alpha k\tilde{u}_{1}
\end{matrix}\right)
\]

From the special \textquotedblleft triangular\textquotedblright{}
form of $\mathcal{A}_{\left(\tilde{u}_{1},0\right)}$, it is clear
that: 
\[
\min\left(\mbox{sp}\left(-\mathcal{A}_{\left(\tilde{u}_{1},0\right)}\right)\right)=\min\left(\lambda_{1,per}\left(-\frac{\mbox{d}^{2}}{\mbox{d}x^{2}}-g_{1}\left[\tilde{u}_{1}\right]\right),\lambda_{1,per}\left(-d\frac{\mbox{d}^{2}}{\mbox{d}x^{2}}-\left(f_{2}\left[0\right]-\alpha k\tilde{u}_{1}\right)\right)\right).
\]

By monotonicity of the periodic principal eigenvalue and $\left(\mathcal{H}_{3}\right)$,
we obtain: 
\[
\lambda_{1,per}\left(-\frac{\mbox{d}^{2}}{\mbox{d}x^{2}}-g_{1}\left[\tilde{u}_{1}\right]\right)>\lambda_{1,per}\left(-\frac{\mbox{d}^{2}}{\mbox{d}x^{2}}-f_{1}\left[\tilde{u}_{1}\right]\right).
\]

For any $k$ large enough, $f_{2}\left[0\right]-\alpha k\tilde{u}_{1}<f_{2}\left[\tilde{u}_{2}\right]$
holds, so that:
\[
\lambda_{1,per}\left(-d\frac{\mbox{d}^{2}}{\mbox{d}x^{2}}-\left(f_{2}\left[0\right]-\alpha k\tilde{u}_{1}\right)\right)>\lambda_{1,per}\left(-d\frac{\mbox{d}^{2}}{\mbox{d}x^{2}}-f_{2}\left[\tilde{u}_{2}\right]\right).
\]

Moreover, from the equation solved by $\tilde{u}_{1}$, $\tilde{u}_{1}$
is actually an eigenfunction for the following eigenvalue: 
\[
\lambda_{1,per}\left(-\frac{\mbox{d}^{2}}{\mbox{d}x^{2}}-f_{1}\left[\tilde{u}_{1}\right]\right)=0.
\]
Similarly, 
\[
\lambda_{1,per}\left(-d\frac{\mbox{d}^{2}}{\mbox{d}x^{2}}-f_{2}\left[\tilde{u}_{2}\right]\right)=0.
\]

Thus:
\[
\lambda_{1,per}\left(-\mathcal{A}_{\left(\tilde{u}_{1},0\right)}\right)>0.
\]

The same proof holds for $\left(0,\tilde{u}_{2}\right)$.
\end{proof}

\subsection{Instability of all periodic coexistence states}

In this subsection, we prove that $\left(\mathcal{M}\right)$ admits
no stable periodic stationary states in $\left\langle \left(0,0\right),\left(\tilde{u}_{1},\tilde{u}_{2}\right)\right\rangle $.

For any $k>k^{\star}$, let: 
\[
S_{k}\subset\mathcal{C}_{per}^{2}\left(\mathbb{R},\mathbb{R}^{2}\right)
\]
 be the set of periodic solutions of the following problem: 
\[
\left\{ \begin{matrix}-u_{1}''=u_{1}f_{1}\left[u_{1}\right]-ku_{1}u_{2}\\
-du_{2}''=u_{2}f_{2}\left[u_{2}\right]-\alpha ku_{1}u_{2}\\
u_{1}\in\left\langle 0,\tilde{u}_{1}\right\rangle \\
u_{2}\in\left\langle 0,\tilde{u}_{2}\right\rangle .
\end{matrix}\right.
\]

Any $\left(u_{1},u_{2}\right)\in S$ is a periodic coexistence state. 

\subsubsection{Basic properties of periodic coexistence states}
\begin{lem}
\label{lem:max_principles_coexistence_states} Let $k>k^{\star}$.
Any $\left(u_{1},u_{2}\right)\in S$ satisfies:
\[
\left\{ \begin{matrix}k\min u_{2}\leq\max f_{1}\left[\max u_{1}\right]\\
\alpha k\min u_{1}\leq\max f_{2}\left[\max u_{2}\right]\\
\min f_{1}\left[\min u_{1}\right]\leq k\max u_{2}\\
\min f_{2}\left[\min u_{2}\right]\leq\alpha k\max u_{1},
\end{matrix}\right.
\]

each extrema being implicitly over $\overline{C}$.
\end{lem}
\begin{proof}
We only prove the first inequality, the three others being proved
similarly.

Let $\overline{x}\in\overline{C}$ such that $u_{1}\left(\overline{x}\right)=\max u_{1}$.
Since $u_{1}\in\mathcal{C}^{2}\left(\mathbb{R}\right)$, $u_{1}''\left(\overline{x}\right)\leq0$,
that is: 
\[
\max u_{1}f_{1}\left[\max u_{1}\right]\geq\max u_{1}ku_{2}\left(\overline{x}\right).
\]

Since $u_{1}>0$, we can divide by $\max u_{1}$. The claimed result
easily follows. 
\end{proof}
\begin{rem*}
This lemma will be used together with $m>0$ to prove that $ku_{1}$
and $ku_{2}$ stay non-zero as $k\to+\infty$. Thus, for the forthcoming
study, it is not sufficient to merely assume that $\lambda_{1,per}\left(-\frac{\mbox{d}^{2}}{\mbox{d}x^{2}}-f_{1}\left[0\right]\right)$
and $\lambda_{1,per}\left(-d\frac{\mbox{d}^{2}}{\mbox{d}x^{2}}-f_{2}\left[0\right]\right)$
are negative (as was done for instance by Dockery and his collaborators
\cite{Dockery_1998}).
\end{rem*}
\begin{prop}
\label{prop:limit_points_coexistence_states} As $k\to+\infty$, the
family $\left(S_{k}\right)_{k>k^{\star}}$ is relatively compact in
$\mathcal{C}_{per}^{0,\beta}\left(\mathbb{R},\mathbb{R}^{2}\right)$.
$\left(0,0\right)$ is one of its limit points. Any other limit point
$\left(u_{1,seg},u_{2,seg}\right)\in\mathcal{C}_{per}^{0,\beta}\left(\mathbb{R},\mathbb{R}^{2}\right)$
is called a periodic segregated state and is such that $\alpha u_{1,seg}-du_{2,seg}$
is a non-zero sign-changing solution in $\mathcal{C}_{per}^{2,\beta}\left(\mathbb{R}\right)$
of the following elliptic equation:
\[
-z''=\eta\left[z\right].
\]
\end{prop}
\begin{proof}
Let $k>k^{\star}$. 

Multiplying by $u_{1,k}$ the first equation of the stationary system
and integrating over $C$ yields easily: 
\begin{eqnarray*}
\|u_{1,k}'\|_{L^{2}\left(C\right)} & \leq & M_{1}\|u_{1,k}\|_{L^{2}\left(C\right)}\\
 & \leq & M_{1}\|\tilde{u}_{1}\|_{L^{2}\left(C\right)},
\end{eqnarray*}
whence, for all $\left(x,y\right)\in C^{2}$:
\[
\left|u_{1,k}\left(x\right)-u_{1,k}\left(y\right)\right|\leq M_{1}\|\tilde{u}_{1}\|_{L^{2}\left(C\right)}\left|x-y\right|^{\nicefrac{1}{2}}.
\]
Moreover, $\|u_{1,k}\|_{L^{\infty}\left(C\right)}\leq\|\tilde{u}_{1}\|_{L^{\infty}\left(C\right)}$,
and therefore $\left(u_{1,k}\right)_{k>k^{\star}}$ is uniformly bounded
in $\mathcal{C}^{0,\nicefrac{1}{2}}\left(C\right)$ and relatively
compact in $\mathcal{C}^{0,\beta}\left(C\right)$. The same proof
holds for $\left(u_{2}\right)_{k>k^{\star}}$. 

Let $\left(u_{1,\infty},u_{2,\infty}\right)\in\mathcal{C}_{per}^{0,\beta}\left(\mathbb{R},\mathbb{R}^{2}\right)$
be a limit point of $\left(S_{k}\right)_{k>k^{\star}}$. There exists
a sequence of periodic coexistence states $\left(\left(u_{1,k},u_{2,k}\right)\right)_{k>k^{\star}}$
whose limit in $\mathcal{C}_{per}^{0,\beta}\left(\mathbb{R},\mathbb{R}^{2}\right)$
is $\left(u_{1,\infty},u_{2,\infty}\right)$. By elliptic regularity
and thanks to the following equation:
\[
-\alpha u_{1,k}''+dv_{2,k}''=\alpha u_{1,k}f_{1}\left[u_{1,k}\right]-u_{2,k}f_{2}\left[u_{2,k}\right],
\]
which holds for any $k>k^{\star}$ and is obtained by linear combination
of the equations of the stationary system, $\left(\alpha u_{1,k}-du_{2,k}\right)$
converge in $\mathcal{C}_{per}^{2,\beta}\left(\mathbb{R}\right)$
to $v=\alpha u_{1,\infty}-du_{2,\infty}\in\mathcal{C}_{per}^{2,\beta}\left(\mathbb{R}\right)$.

Multiplying by a test function $\varphi\in\mathcal{D}\left(\mathbb{R}\right)$
the equation defining $u_{1,k}$, integrating and dividing by $k$,
we obtain easily that $\left(u_{1,k}u_{2,k}\right)$ converges as
$k\to+\infty$ in $\mathcal{D}'\left(\mathbb{R}\right)$ to $0$.
Hence $u_{1,\infty}u_{2,\infty}=0$ and then $\alpha u_{1,\infty}=v^{+}$
and $du_{2,\infty}=v^{-}$. In particular, $v$ satisfies as claimed:
\[
-v''=\eta\left[v\right]
\]

Let: 
\[
C_{1}=\left\{ x\in C\ |\ v\left(x\right)>0\right\} ,
\]
\[
C_{2}=\left\{ x\in C\ |\ v\left(x\right)<0\right\} ,
\]
\[
\Gamma=\left\{ x\in C\ |\ v\left(x\right)=0\right\} ,
\]
so that: 
\[
C\subset C_{1}\cup C_{2}\cup\Gamma\subset\overline{C}.
\]

Exactly four cases are a priori possible:
\begin{enumerate}
\item $C_{1}=C$: then by continuity $v=\alpha u_{1,\infty}$ in $\overline{C}$
whereas $u_{2,\infty}=0$ in $\overline{C}$, hence $u_{1,\infty}\in\mathcal{C}_{per}^{2,\beta}\left(\mathbb{R}\right)$
is a non-negative non-zero solution of
\[
-u_{1,\infty}''=u_{1,\infty}f_{1}\left[u_{1,\infty}\right]
\]
 in $\mathbb{R}$, and eventually by the elliptic strong minimum principle
$u_{1,\infty}\gg0$, meaning that $u_{1,\infty}=\tilde{u}_{1}$, and
$C_{2}=\Gamma=\emptyset$;
\item $C_{2}=C$: then similarly $C_{1}=\Gamma=\emptyset$, $u_{1,\infty}=0$
and $u_{2,\infty}=\tilde{u}_{2}$;
\item $C_{1}\neq\emptyset$ and $C_{2}\neq\emptyset$.
\item $C_{1}=\emptyset$ and $C_{2}=\emptyset$: $\Gamma=C$, $u_{1,\infty}$
and $v_{2,\infty}$ are uniformly $0$;
\end{enumerate}
It is easily seen that Lemma \ref{lem:max_principles_coexistence_states}
excludes the cases $1$ (use the second inequality) and $2$ (use
the first inequality).
\end{proof}
\begin{prop}
\label{prop:high_frequency_consequence}The following set equalities
hold:
\[
\left\{ z\in\mathcal{C}_{per}^{2}\left(\mathbb{R}\right)\ |\ -z''=\gamma\left[z\right]\right\} =\left\{ 0\right\} ,
\]
\[
\left\{ z\in\mathcal{C}_{per}^{2}\left(\mathbb{R}\right)\ |\ -z''=\eta\left[z\right]\right\} =\left\{ -d\tilde{u}_{2},0,\alpha\tilde{u}_{1}\right\} .
\]
\end{prop}
\begin{proof}
In the $\gamma$ case, solutions of constant sign are excluded by:
\[
\lambda_{1,per}\left(-\frac{\mbox{d}^{2}}{\mbox{d}x^{2}}-f_{1}\left[0\right]\right)<0,
\]
\[
\lambda_{1,per}\left(-d\frac{\mbox{d}^{2}}{\mbox{d}x^{2}}-f_{2}\left[0\right]\right)<0.
\]

In the $\eta$ case, solutions of constant sign are unique (see Berestycki\textendash Hamel\textendash Roques
\cite{Berestycki_Ham_1}) and are exactly $\alpha\tilde{u}_{1}$ and
$-d\tilde{u}_{2}$. It only remains to prove that non-zero sign-changing
solutions are excluded, and up to a shift of $C$ it suffices to prove
that non-zero sign-changing solutions which are equal to $0$ at $0$
and $L$ are excluded. 

For any $x\in\mathbb{R}$, any $f\in\mathcal{C}_{per}^{0}\left(\mathbb{R},\left[m,M\right]\right)$
and any $\delta\in\left\{ 1,d\right\} $, let $R\left(x,f,\delta\right)>0$
such that: 
\[
\lambda_{1,Dir}\left(-\delta\frac{\mbox{d}^{2}}{\mbox{d}x^{2}}-f,B\left(x,R\left(x,f,\delta\right)\right)\right)=0.
\]
Since the following function: 
\[
R\mapsto\lambda_{1,Dir}\left(-\delta\frac{\mbox{d}^{2}}{\mbox{d}x^{2}}-f,B\left(x,R\right)\right)
\]
 is continuous, decreasing and has positive and negative values (its
limits as $R\to0$ or $R\to+\infty$ are respectively $+\infty$ and
$\lambda_{1,per}\left(-\delta\frac{\mbox{d}^{2}}{\mbox{d}x^{2}}-f\right)<0$,
as proved in \cite{Berestycki_Ham_1}), $R\left(x,f,\delta\right)$
is uniquely defined. Since $\lambda_{1,Dir}\left(-\delta\frac{\mbox{d}^{2}}{\mbox{d}x^{2}}-f,B\left(x,R\right)\right)$
is non-increasing with respect to $f$ and decreasing with respect
to $R$, it is easy to check that $f\mapsto R\left(x,f,\delta\right)$
is non-increasing. 

Remark that $R\left(x,f,\delta\right)$ and $\lambda_{1,Dir}\left(-\delta\frac{\mbox{d}^{2}}{\mbox{d}x^{2}}-f,B\left(x,R\left(x,f,\delta\right)\right)\right)$
do not depend on $x$ if $f$ does not depend on $x$. Remark that,
in such a case, $R\left(0,f,\delta\right)$ can be easily determined
analytically and is equal to $\frac{\pi}{2}\sqrt{\frac{\delta}{f}}$.

With these notations, $\left(\mathcal{H}_{freq}\right)$ means: 
\[
L<2\left(R\left(0,M_{1},1\right)+R\left(0,M_{2},d\right)\right).
\]

Let $z$ be a solution of $-z''=\gamma\left[z\right]$ or a solution
of $-z''=\eta\left[z\right]$. Let: 
\[
C_{+}=z^{-1}\left(\left(0,+\infty\right)\right)\cap C,
\]
\[
C_{-}=z^{-1}\left(\left(-\infty,0\right)\right)\cap C.
\]
Assume by contradiction that both are non-empty. Let $n$ be the number
of zeros of $z$ in $C$. Then:
\begin{itemize}
\item in virtue of the Hopf lemma, of:
\[
\min\left(\min_{x\in\overline{C}}R\left(x,f_{1}\left[0\right],1\right),\min_{x\in\overline{C}}R\left(x,f_{2}\left[0\right],d\right)\right)>0
\]
 and of the continuity of $z$, $n$ is finite and odd, say $n=2p+1$
with $p$ a non-negative integer, and $C_{+}$ and $C_{-}$ both have
precisely $p+1$ connected components, each of them being a one-dimensional
ball (that is an interval); let $\left(x_{i}^{+}\right)_{1\leq i\leq p+1}$
(resp. $\left(x_{i}^{-}\right)_{1\leq i\leq p+1}$) be the ordered
centers of the connected components of $C_{+}$ (resp. $C_{-}$); 
\item in the $\gamma$ case: 
\begin{eqnarray*}
\left|C_{+}\right| & = & 2\sum_{i=1}^{p+1}R\left(x_{i}^{+},f_{1}\left[0\right],1\right)\\
 & \geq & 2\sum_{i=1}^{p+1}R\left(x_{i}^{+},M_{1},1\right)\\
 & \geq & 2\left(p+1\right)R\left(0,M_{1},1\right)\\
 & \geq & 2R\left(0,M_{1},1\right),
\end{eqnarray*}
and similarly: 
\begin{eqnarray*}
\left|C_{-}\right| & = & 2\sum_{i=1}^{p+1}R\left(x_{i}^{-},f_{2}\left[0\right],d\right)\\
 & \geq & 2R\left(0,M_{2},d\right),
\end{eqnarray*}
 whence we get the contradiction;
\item in the $\eta$ case: 
\begin{align*}
\left|C_{+}\right| & =2\sum_{i=1}^{p+1}R\left(x_{i}^{+},f_{1}\left[\frac{z}{\alpha}\right],1\right)\\
 & \geq2\sum_{i=1}^{p+1}R\left(x_{i}^{+},f_{1}\left[0\right],1\right),
\end{align*}
\begin{align*}
\left|C_{-}\right| & =2\sum_{i=1}^{p+1}R\left(x_{i}^{-},f_{2}\left[-\frac{z}{d}\right],d\right)\\
 & \geq2\sum_{i=1}^{p+1}R\left(x_{i}^{-},f_{2}\left[0\right],d\right)
\end{align*}
 yield a similar contradiction.
\end{itemize}
\end{proof}
\begin{cor}
\label{cor:only_limit_stationary_state} Any family $\left(u_{1,k},u_{2,k}\right)_{k>k^{\star}}$
of periodic coexistence states converges in $\mathcal{C}_{per}^{0,\beta}\left(\mathbb{R},\mathbb{R}^{2}\right)$
as $k\to+\infty$ to $\left(0,0\right)$.
\end{cor}
\begin{rem*}
This result has a very natural interpretation from an ecological point
of view: if the wavelength of the distribution of resources is small
enough, or if the resources are rare enough even in the most favorable
areas, the species are not able to settle periodically in a favorable
habitat smaller than the wavelength. Either one of them is strong
enough to overcome unfavorable areas while eliminating the competitor
and then it settles in the whole habitat, either both go extinct.
Basically, at a given average intrinsic growth rate, the more fragmented
the habitat is, the higher the chances of extinction are. 
\end{rem*}
\begin{lem}
\label{lem:ku/kv_bounded}There exists $R_{1}\in\left(0,+\infty\right)$
and $R_{2}\in\left(R_{1},+\infty\right)$ such that, provided $k^{\star}$
is large enough, for any $k>k^{\star}$ and any $\left(u_{1,k},u_{2,k}\right)\in S_{k}$:
\[
R_{1}\leq\frac{\|u_{2,k}\|_{L^{\infty}\left(C\right)}}{\alpha\|u_{1,k}\|_{L^{\infty}\left(C\right)}}\leq R_{2}.
\]
\end{lem}
\begin{rem*}
Proof inspired by Dancer\textendash Du \cite[Lemma 2.1]{Dancer_Du_1994}. 
\end{rem*}
\begin{proof}
By contradiction, assume that there exists a sequence of periodic
coexistence states $\left(\left(u_{1,k},u_{2,k}\right)\right)_{k>k^{\star}}$
such that $\left(\frac{\|u_{2,k}\|_{L^{\infty}\left(C\right)}}{\alpha\|u_{1,k}\|_{L^{\infty}\left(C\right)}}\right)_{k>k^{\star}}$
is neither bounded from above nor from below by a positive constant.
By symmetry, we can assume without loss of generality that it is not
bounded from below by a positive constant. Up to extraction, $\frac{\|u_{2,k}\|_{L^{\infty}\left(C\right)}}{\alpha\|u_{1,k}\|_{L^{\infty}\left(C\right)}}\to0$
as $k\to+\infty$. 

Suppose first that $\left(\alpha k\|u_{1,k}\|_{L^{\infty}\left(C\right)}\right)_{k>k^{\star}}$
is bounded. Necessarily, $k\|u_{2,k}\|_{L^{\infty}\left(C\right)}\to0$
as $k\to+\infty$.

For any non-negative $f\in\mathcal{C}\left(\mathbb{R},\mathbb{R}\right)$,
the following problem: 
\[
-z''=zf_{1}\left[z\right]-zf
\]
 with periodicity conditions has a unique positive periodic solution
$z_{f}$ if and only if: 
\[
\lambda_{1,per}\left(-\frac{\mbox{d}^{2}}{\mbox{d}x^{2}}-\left(f_{1}-f\right)\right)<0
\]
 (see Berestycki\textendash Hamel\textendash Roques \cite{Berestycki_Ham_1}).
Moreover, $z_{f}$ depends continuously on $f$ as a map from $\mathcal{C}_{per}\left(C\right)$
into itself (see Berestycki\textendash Rossi \cite{Berestycki_Ros}).
Hence $u_{1,k}=z_{ku_{2,k}}\to z_{0}$ as $k\to+\infty$, where $z_{0}$
solves: 
\[
-z_{0}''=z_{0}f_{1}\left[z_{0}\right]
\]
 with periodicity conditions (that is $u\left[0\right]=\tilde{u}_{1}$).
Since $k\|\tilde{u}_{1}\|_{L^{\infty}\left(C\right)}\to+\infty$,
we get a contradiction.

Hence $\left(\alpha k\|u_{1,k}\|_{L^{\infty}\left(C\right)}\right)_{k>k^{\star}}$
is unbounded. Up to extraction, we can assume that $k\|u_{1,k}\|_{L^{\infty}\left(C\right)}\to+\infty$. 

For any $k>k^{\star}$, let $\hat{u}_{1,k}=\frac{u_{1,k}}{\|u_{1,k}\|_{L^{\infty}\left(C\right)}}$,
$\hat{u}_{2,k}=\frac{u_{2,k}}{\|u_{2,k}\|_{L^{\infty}\left(C\right)}}$.
Clearly, $\left(\hat{u}_{1,k},\hat{u}_{2,k}\right)$ satisfies:
\[
\left\{ \begin{matrix}-\hat{u}_{1,k}''=\hat{u}_{1,k}f_{1}\left[\|u_{1,k}\|_{L^{\infty}\left(C\right)}\hat{u}_{1,k}\right]-k\|u_{2,k}\|_{L^{\infty}\left(C\right)}\hat{u}_{1,k}\hat{u}_{2,k}\\
-d\hat{u}_{2,k}''=\hat{u}_{2,k}f_{2}\left[\|u_{2,k}\|_{L^{\infty}\left(C\right)}\hat{u}_{2,k}\right]-\alpha k\|u_{1,k}\|_{L^{\infty}\left(C\right)}\hat{u}_{1,k}\hat{u}_{2,k}.
\end{matrix}\right.
\]

From there, it follows with the same estimates as in the proof of
Proposition \ref{prop:limit_points_coexistence_states} that $\hat{u}_{1,k}$
and $\hat{u}_{2,k}$ converge up to extraction in $\mathcal{C}_{per}^{0,\beta}\left(\mathbb{R}\right)$.
Let $\hat{u}_{1,\infty}$ and $\hat{u}_{2,\infty}$ be their limits;
for any $i\in\left\{ 1,2\right\} $ $\|\hat{u}_{i,\infty}\|_{L^{\infty}\left(C\right)}=1$,
hence $u_{i,\infty}\neq0$. 

Then, we consider the system above in $\mathcal{D}'\left(C\right)$.
Let $\varphi\in\mathcal{D}\left(C\right)$ and use it as a test function.
On the second line, we see that, since: 
\[
\int\left(d\hat{u}_{2,k}''+\hat{u}_{2,k}f_{2}\left[\|u_{2,k}\|_{L^{\infty}\left(C\right)}\hat{u}_{2,k}\right]\right)\varphi
\]
 is $k$-uniformly bounded, the same is true of: 
\[
\int\alpha k\|u_{1,k}\|_{L^{\infty}\left(C\right)}\hat{u}_{1,k}\hat{u}_{2,k}\varphi.
\]

Thus: 
\[
\int k\|u_{2,k}\|_{L^{\infty}\left(C\right)}\hat{u}_{1,k}\hat{u}_{2,k}\varphi=\frac{\|u_{2,k}\|_{L^{\infty}\left(C\right)}}{\alpha\|u_{1,k}\|_{L^{\infty}\left(C\right)}}\int\left(\alpha k\|u_{1,k}\|_{L^{\infty}\left(C\right)}\hat{u}_{1,k}\hat{u}_{2,k}\varphi\right)\to0
\]

Therefore, considering the first line, we see that, by dominated convergence,
the limit satisfies in the distributional sense:
\[
-\hat{u}_{1,\infty}''=\hat{u}_{1,\infty}f_{1}\left[\|u_{1,\infty}\|_{L^{\infty}\left(C\right)}\hat{u}_{1,\infty}\right].
\]

Since $\hat{u}_{1,\infty}$ is in $\mathcal{C}_{per}^{0,\beta}\left(\mathbb{R}\right)$,
it is actually a solution in $\mathcal{C}_{per}^{2,\beta}\left(\mathbb{R}\right)$
by classical elliptic regularity. In virtue of the elliptic strong
minimum principle, $\hat{u}_{1,\infty}\gg0$. But it is also true,
using the same arguments as before, that $\hat{u}_{1,\infty}\hat{u}_{2,\infty}=0$,
hence $\hat{u}_{2,\infty}=0$, which is indeed a contradiction.
\end{proof}
\begin{lem}
\label{lem:(ku,kv)_bounded} Let $\left(\left(u_{1,k},u_{2,k}\right)\right)_{k>k^{\star}}$
be a sequence of periodic coexistence states. Then $\left(\left(ku_{1,k},ku_{2,k}\right)\right)_{k>k^{\star}}$
is $k$-uniformly bounded in $L^{\infty}\left(C\right)$.
\end{lem}
\begin{proof}
From Lemma \ref{lem:ku/kv_bounded}, it suffices to assume that there
exists a sequence $\left(\left(u_{1},u_{2}\right)\right)_{k>k^{\star}}$
such that $k\|u_{1,k}\|_{L^{\infty}\left(C\right)}\to+\infty$ as
$k\to+\infty$ and to get a contradiction.

With the same notations as in the proof of Lemma \ref{lem:ku/kv_bounded},
up to extraction we can assume that $\hat{u}_{1,k}\to\hat{u}_{1,\infty}$
and $\hat{u}_{2,k}\to\hat{u}_{2,\infty}$ in $\mathcal{C}_{per}^{0,\beta}\left(\mathbb{R}\right)$.
We have for any $i\in\left\{ 1,2\right\} $ $\|\hat{u}_{i,\infty}\|_{L^{\infty}\left(C\right)}=1$,
hence $u_{i,\infty}\neq0$. Considering the limit of the equation
satisfied by $\hat{u}_{2,k}$ in $\mathcal{D}'\left(C\right)$ shows
that $\hat{u}_{1,\infty}\hat{u}_{2,\infty}=0$. Thanks to Lemma \ref{lem:ku/kv_bounded},
up to extraction, we can assume that there exists $l>0$ such that
$\frac{\alpha\|u_{1,k}\|_{L^{\infty}\left(C\right)}}{\|u_{2,k}\|_{L^{\infty}\left(C\right)}}\to l$.
Moreover, considering the equation satisfied by $\hat{u}_{1,k}$ in
$\mathcal{D}'\left(C\right)$ shows that, for any $\varphi\in\mathcal{D}\left(C\right)$:
\[
\int k\|u_{2,k}\|_{L^{\infty}\left(C\right)}\hat{u}_{1,k}\hat{u}_{2,k}\varphi
\]
 is $k$-uniformly bounded.

Multiplying the equation defining $\hat{u}_{1,k}$ by $l$ and subtracting
from it the equation defining $\hat{u}_{2,k}$ yields:
\begin{eqnarray*}
-l\hat{u}_{1,k}''+d\hat{u}_{2,k}'' & = & l\hat{u}_{1,k}f_{1}\left[\|u_{1,k}\|_{L^{\infty}\left(C\right)}\hat{u}_{1,k}\right]-\hat{u}_{2,k}f_{2}\left[\|u_{2,k}\|_{L^{\infty}\left(C\right)}\hat{u}_{2,k}\right]\\
 &  & +\left(\frac{\alpha\|u_{1,k}\|_{L^{\infty}\left(C\right)}}{\|u_{2,k}\|_{L^{\infty}\left(C\right)}}-l\right)k\|u_{2,k}\|_{L^{\infty}\left(C\right)}\hat{u}_{1,k}\hat{u}_{2,k}.
\end{eqnarray*}

Considering it in $\mathcal{D}'\left(C\right)$, passing to the limit
(with, in virtue of Corollary \ref{cor:only_limit_stationary_state},
$\|u_{i,k}\|_{L^{\infty}\left(C\right)}\to0$) and defining $v=l\hat{u}_{1,\infty}-d\hat{u}_{2,\infty}$,
it becomes:
\[
-v''=\gamma\left[v\right].
\]

By classical elliptic regularity, $v$ is actually a solution in $\mathcal{C}_{per}^{2,\beta}\left(\mathbb{R}\right)$.
Then Proposition \ref{prop:high_frequency_consequence} implies $l\hat{u}_{1,\infty}=d\hat{u}_{2,\infty}$,
but together with $\hat{u}_{1,\infty}\hat{u}_{2,\infty}=0$ and the
fact that the pair $\left(u_{1,\infty},u_{2,\infty}\right)$ is non-zero,
this is a contradiction.
\end{proof}
\begin{lem}
\label{lem:ku,kv_bounded_from_below} Provided $k^{\star}$ is large
enough, the following lower bound holds:

\[
\inf_{k>k^{\star}}\inf_{\left(u_{1},u_{2}\right)\in S_{k}}\min\left\{ \min_{\overline{C}}\left(ku_{1}\right),\min_{\overline{C}}\left(ku_{2}\right)\right\} >0
\]
\end{lem}
\begin{proof}
Let $\left(\left(u_{1,k},u_{2,k}\right)\right)_{k>k^{\star}}$. For
any $i\in\left\{ 1,2\right\} $ and any $k>k^{\star}$, let $U_{i,k}=ku_{i,k}$.
$\left(U_{1,k},U_{2,k}\right)$ satisfies the following system:
\[
\left\{ \begin{matrix}-U_{1,k}''=U_{1,k}f_{1}\left[\frac{U_{1,k}}{k}\right]-U_{1,k}U_{2,k}\\
-dU_{2,k}''=U_{2,k}f_{2}\left[\frac{U_{2,k}}{k}\right]-\alpha U_{1,k}U_{2,k}.
\end{matrix}\right.
\]

Since $U_{1,k}$ and $U_{2,k}$ are $k$-uniformly bounded in $L^{\infty}\left(C\right)$
in virtue of Lemma \ref{lem:(ku,kv)_bounded}, we can prove with the
same arguments as before that, for any $i\in\left\{ 1,2\right\} $
and up to extraction, $U_{i,k}$ converges in $\mathcal{C}_{per}^{0,\beta}\left(\mathbb{R}\right)$
to some $U_{i,\infty}\geq0$, and by Lemma \ref{lem:max_principles_coexistence_states}
(third and fourth inequalities), $U_{i,\infty}\neq0$. The limits
satisfy the remarkable following system:
\[
\left\{ \begin{matrix}-U_{1,\infty}''=U_{1,\infty}f_{1}\left[0\right]-U_{1,\infty}U_{2,\infty}\\
-dU_{2,\infty}''=U_{2,\infty}f_{2}\left[0\right]-\alpha U_{1,\infty}U_{2,\infty}.
\end{matrix}\right.
\]

At first this system is to be understood in the distributional sense,
but once more thanks to classical elliptic regularity $U_{1,\infty}$
and $U_{2,\infty}$ are actually in $\mathcal{C}_{per}^{2,\beta}\left(\mathbb{R}\right)$.
Thanks to the elliptic strong minimum principle, for any $i\in\left\{ 1,2\right\} $,
$U_{i,\infty}\gg0$.

In $C$, $-\frac{U_{1,\infty}''}{U_{1,\infty}}=f_{1}\left[0\right]-U_{2,\infty}\leq M_{1}$.
Integration over $C$ yields:
\[
\int_{C}f_{1}\left[0\right]=-\int_{C}\left|\frac{U_{1,\infty}'}{U_{1,\infty}}\right|^{2}+\int_{C}U_{2,\infty}\leq\int_{C}U_{2,\infty}.
\]

Similarly,
\[
\int_{C}f_{2}\left[0\right]\leq\int_{C}U_{1,\infty}.
\]

Then $\left(\mathcal{H}_{2}\right)$ shows that $\left(U_{1,\infty},U_{2,\infty}\right)$
is at positive distance of the origin in $L^{1}\left(C\right)$, and
then in $L^{\infty}\left(C\right)$ by classical embeddings. Harnack\textquoteright s
inequality yields eventually that $\min\left(\min\limits _{\overline{C}}\left(U_{1,\infty}\right),\min\limits _{\overline{C}}\left(U_{2,\infty}\right)\right)$
is bounded from below by a real number $\epsilon>0$. By uniform convergence
and provided $k^{\star}$ is large enough, the infimum of the sequence
$\left(\min\left\{ \min\limits _{\overline{C}}\left(ku_{1,k}\right),\min\limits _{\overline{C}}\left(ku_{2,k}\right)\right\} \right)_{k>k^{\star}}$
is greater than, say, $\frac{3\epsilon}{4}$. This $\epsilon$ depends
on $m$, $C$, but neither on the limit point $\left(U_{1,\infty},U_{2,\infty}\right)$
nor on the choice of a convergent subsequence of $\left(\left(u_{1},u_{2}\right)\right)_{k>k^{\star}}$,
whence the bound holds for any convergent subsequence of $\left(\left(u_{1},u_{2}\right)\right)_{k>k^{\star}}$.
Furthermore, the bound does not depend on the choice of the sequence
$\left(\left(u_{1},u_{2}\right)\right)_{k>k^{\star}}$ itself, whence
it holds for any convergent subsequence of any sequence.

The conclusion on the whole set is a standard compactness argument%
\begin{comment}
: assuming by contradiction that there is an infinite subset of $\mathfrak{S}\left(\mathfrak{K}\right)$,
indexed on an unbounded subset of $\mathfrak{K}$, in which $\left(\min_{\overline{C}}\left(ku_{1}\right),\min_{\overline{C}}\left(ku_{2}\right)\right)$
is not bounded from below by, say, $\frac{\epsilon}{2}$, we extract
a convergent subsequence and contradict the previous result. Therefore,
there is only finitely many elements of $\mathfrak{S}\left(\mathfrak{K}\right)$
such that $\left(\min\left\{ \min_{\overline{C}}\left(ku_{1}\right),\min_{\overline{C}}\left(ku_{2}\right)\right\} \right)$
is not bounded from below by $\frac{\epsilon}{2}$ and we can exclude
these elements provided $k^{\star}$is large enough without loss of
generality
\end{comment}
.
\end{proof}

\subsubsection{Instability of periodic coexistence states close to $\left(0,0\right)$}
\begin{lem}
\label{lem:A_strongly_positive}Provided $k^{\star}$ is large enough,
for any $\left(u_{1},u_{2}\right)\in S$, the differential operator
$\mathcal{A}_{\left(u_{1},u_{2}\right)}:\mathcal{C}_{per}^{2}\left(\mathbb{R}\right)\to\mathcal{C}_{per}\left(\mathbb{R}\right)$
defined as: 
\[
\mathcal{A}_{\left(u_{1},u_{2}\right)}=\left(\begin{matrix}\frac{\mbox{d}^{2}}{\mbox{d}x^{2}}+g_{1}\left[u_{1}\right]-ku_{2} & ku_{1}\\
\alpha ku_{2} & d\frac{\mbox{d}^{2}}{\mbox{d}x^{2}}+g_{2}\left[u_{2}\right]-\alpha ku_{1}
\end{matrix}\right)
\]
 is strongly positive.
\end{lem}
\begin{proof}
It is well-known that $\mathcal{A}_{\left(u_{1},u_{2}\right)}$ is
strongly positive (i.e. satisfies the strong minimum principle) if
there exists a pair of positive functions whose image by $-\mathcal{A}_{\left(u_{1},u_{2}\right)}$
is itself non-negative (see for instance Figueiredo\textendash Mitidieri
\cite{Figueiredo_Mit}). From $\left(\mathcal{H}_{1}\right)$, if
$k$ is large enough, there exists a constant $R>0$ which depends
only on $x\mapsto\partial_{1}f_{1}\left(0,x\right)$ and $x\mapsto\partial_{1}f_{2}\left(0,x\right)$
 such that: 
\[
\left\{ \begin{matrix}\partial_{1}f_{1}\left[u_{1}\right]\in\left[-R,0\right]\,\\
\partial_{1}f_{2}\left[u_{2}\right]\in\left[-R,0\right].
\end{matrix}\right.
\]

From here, it is easy to check that, up to extraction and using the
notations of the proof of Lemma \ref{lem:ku,kv_bounded_from_below},
\[
-\mathcal{A}_{\left(u_{1,k},u_{2,k}\right)}\left(\begin{matrix}U_{1,\infty}\\
U_{2,\infty}
\end{matrix}\right)\to\left(\begin{matrix}U_{1,\infty}U_{2,\infty}\\
\alpha U_{1,\infty}U_{2,\infty}
\end{matrix}\right)
\]
uniformly in $C$ as $k\to+\infty$. 

This limit being positive, thanks to standard compactness arguments,
we get indeed the claimed statement. 
\end{proof}
\begin{prop}
\label{prop:unstable_s_s} For any $k>k^{\star}$, any $\left(u_{1},u_{2}\right)\in S$
is unstable.
\end{prop}
\begin{proof}
Thanks to Mora\textquoteright s theorem \cite{Mora_1983}, we know
that $\left(u_{1},u_{2}\right)$ is unstable if the principal eigenvalue
of the elliptic part of the monotone problem $\left(\mathcal{M}\right)$
linearized at $\left(u_{1},J\left(u_{2}\right)\right)$ is negative.
It is easy to verify that the linearized operator is in fact:
\[
\mathcal{A}_{\left(u_{1},u_{2}\right)}=\left(\begin{matrix}\frac{\mbox{d}^{2}}{\mbox{d}x^{2}}+g_{1}\left[u_{1}\right]-ku_{2} & ku_{1}\\
\alpha ku_{2} & d\frac{\mbox{d}^{2}}{\mbox{d}x^{2}}+g_{2}\left[u_{2}\right]-\alpha ku_{1}
\end{matrix}\right)
\]

$\mathcal{A}_{\left(u_{1},u_{2}\right)}$ being strongly positive
(see Lemma \ref{lem:A_strongly_positive}), it is injective and, up
to a restriction of its codomain, it is invertible. Krein\textendash Rutman\textquoteright s
theorem and a well-known routine involving the compact canonical embedding
$\mathcal{C}^{2,\beta}\left(C\right)\hookrightarrow\mathcal{C}_{loc}^{0,\beta}\left(C\right)$
prove the existence of the periodic principal eigenvalue $\lambda_{1,per}\left(-\mathcal{A}_{\left(u_{1},u_{2}\right)}\right)$. 

Now, we have to prove that $\lambda_{1,per}\left(-\mathcal{A}_{\left(u_{1},u_{2}\right)}\right)<0$.
Recall the following characterization from Krein\textendash Rutman\textquoteright s
theorem:
\[
\lambda_{1,per}\left(-\mathcal{A}_{\left(u_{1},u_{2}\right)}\right)=\inf\left\{ \lambda\in\mathbb{R}\ |\ \exists\varphi\in\mathcal{C}_{per}^{2}\left(\mathbb{R},\left(0,+\infty\right)^{2}\right)\ \left(-\mathcal{A}_{\left(u_{1},u_{2}\right)}-\lambda\right)\varphi\leq0\ \mbox{in}\ \mathbb{R}\right\} .
\]

Therefore, we only need to find some $\lambda<0$ and some $\varphi\in\mathcal{C}_{per}^{2}\left(\mathbb{R},\left(0,+\infty\right)^{2}\right)$
satisfying: 
\[
\left(-\mathcal{A}_{\left(u_{1},u_{2}\right)}-\lambda\right)\varphi\leq0.
\]

Using $\left(\mathcal{H}_{1}\right)$, it is easy to check that there
exists a constant $R>0$ which depends only on $x\mapsto\partial_{1}f_{1}\left(0,x\right)$
and $x\mapsto\partial_{1}f_{2}\left(0,x\right)$ such that:

\begin{eqnarray*}
\left(-\mathcal{A}_{\left(u_{1},u_{2}\right)}\right)\left(\begin{matrix}u_{1}\\
u_{2}
\end{matrix}\right) & = & \left(\begin{matrix}-u_{1}^{2}\partial_{1}f_{1}\left[u_{1}\right]-ku_{1}u_{2}\\
-u_{2}^{2}\partial_{1}f_{2}\left[u_{2}\right]-\alpha ku_{1}u_{2}
\end{matrix}\right)\\
 & \leq & \left(\begin{matrix}\left(Ru_{1}-ku_{2}\right)u_{1}\\
\left(Ru_{2}-\alpha ku_{1}\right)u_{2}
\end{matrix}\right)\\
 & \leq & -\min\left\{ \min_{\overline{C}}\left(ku_{2}-Ru_{1}\right),\min_{\overline{C}}\left(\alpha ku_{1}-Ru_{2}\right)\right\} \left(\begin{matrix}u_{1}\\
u_{2}
\end{matrix}\right).
\end{eqnarray*}

In virtue of Lemma \ref{lem:ku,kv_bounded_from_below}, provided $k^{\star}$
is large enough, for any $K>k^{\star}$ and any $\left(u_{1,K},u_{2,K}\right)\in S_{K}$:
\[
\min\left\{ \min_{\overline{C}}\left(Ku_{2,K}-Ru_{1,K}\right),\min_{\overline{C}}\left(\alpha Ku_{1,K}-Ru_{2,K}\right)\right\} >0.
\]

Consequently it holds for $k$ and $\left(u_{1},u_{2}\right)$.

Now, if we define $\lambda$ as $-\min\left\{ \min\limits _{\overline{C}}\left(ku_{2}-Ru_{1}\right),\min\limits _{\overline{C}}\left(\alpha ku_{1}-Ru_{2}\right)\right\} $
and $\varphi$ as $\left(u_{1},u_{2}\right)$, it is obvious that
$\left(-\mathcal{A}_{\left(u_{1},u_{2}\right)}-\lambda\right)\varphi\leq0$.
Therefore, $\left(u_{1},u_{2}\right)$ is unstable.
\end{proof}

\subsection{Counter-propagation}

In this subsection, we prove the so-called counter-propagation hypothesis.
Let us recall from Fang\textendash Zhao \cite{Fang_Zhao_2011} that,
since every intermediate periodic stationary state is unstable (Proposition
\ref{prop:unstable_s_s}), their set is totally unordered.
\begin{prop}
Let $k>k^{\star}$ and $\left(u_{1},u_{2}\right)\in S$. 

Let $c_{+}^{\star}\left(\left(u_{1},\tilde{u}_{2}-u_{2}\right),\left(\tilde{u}_{1},\tilde{u}_{2}\right)\right)\in\mathbb{R}$
and $c_{-}^{\star}\left(\left(u_{1},\tilde{u}_{2}-u_{2}\right),\left(0,0\right)\right)\in\mathbb{R}$
be the spreading speeds associated with front-like initial data connecting
respectively $\left(\tilde{u}_{1},\tilde{u}_{2}\right)$ to $\left(u_{1},\tilde{u}_{2}-u_{2}\right)$
and $\left(u_{1},\tilde{u}_{2}-u_{2}\right)$ to $\left(0,0\right)$. 

Then:
\[
c_{+}^{\star}\left(\left(u_{1},\tilde{u}_{2}-u_{2}\right),\left(\tilde{u}_{1},\tilde{u}_{2}\right)\right)+c_{-}^{\star}\left(\left(u_{1},\tilde{u}_{2}-u_{2}\right),\left(0,0\right)\right)>0.
\]
\end{prop}
\begin{rem*}
At least formally, since $\left(u_{1},u_{2}\right)$ vanishes as $k\to+\infty$,
we have:
\[
c_{+}^{\star}\left(\left(u_{1},\tilde{u}_{2}-u_{2}\right),\left(\tilde{u}_{1},\tilde{u}_{2}\right)\right)\to c_{+}^{\star}\left(\left(0,\tilde{u}_{2}\right),\left(\tilde{u}_{1},\tilde{u}_{2}\right)\right),
\]
\[
c_{-}^{\star}\left(\left(u_{1},\tilde{u}_{2}-u_{2}\right),\left(0,0\right)\right)\to c_{-}^{\star}\left(\left(0,\tilde{u}_{2}\right),\left(0,0\right)\right).
\]

It is easily seen that the first limit is in fact the spreading speed
of the scalar KPP pulsating front connecting $\tilde{u}_{1}$ to $0$
for the equation $\partial_{t}u_{1}-\partial_{xx}u_{1}=u_{1}f_{1}\left[u_{1}\right]$
whereas the second one is in fact the spreading speed of the scalar
KPP pulsating front connecting $\tilde{u}_{2}$ to $0$ for the equation
$\partial_{t}u_{2}-d\partial_{xx}u_{2}=u_{2}f_{2}\left[u_{2}\right]$.
These limiting speeds are both positive. Hence, heuristically, we
expect that both $c_{+}^{\star}\left(\left(u_{1},\tilde{u}_{2}-u_{2}\right),\left(\tilde{u}_{1},\tilde{u}_{2}\right)\right)$
and $c_{-}^{\star}\left(\left(u_{1},\tilde{u}_{2}-u_{2}\right),\left(0,0\right)\right)$
are positive whenever $k$ is large enough, and this is indeed what
we will prove. 
\end{rem*}
\begin{proof}
Let $k>k^{\star}$, $\left(u_{1},u_{2}\right)\in S$, $\mathcal{A}_{\left(u_{1},u_{2}\right)}$
be the associated linear elliptic operator defined as in Lemma \ref{lem:A_strongly_positive},
$t>0$, $Q_{t}$ be the semiflow associated with $\left(\mathcal{M}\right)$
and $Q_{t}^{u,lin}$ be the linear semiflow associated with $\partial_{t}-\mathcal{A}_{\left(u_{1},u_{2}\right)}$.
We intend to use Weinberger\textquoteright s theory \cite[Theorem 2.4]{Weinberger_200}
in order to establish that:
\[
c_{+}^{\star}\left(\left(u_{1},\tilde{u}_{2}-u_{2}\right),\left(\tilde{u}_{1},\tilde{u}_{2}\right)\right)\geq\inf_{\mu>0}\frac{-\lambda_{1,per}\left(-\mu^{2}\text{diag}\left(1,d\right)-\mathcal{A}_{\left(u_{1},u_{2}\right)}\right)}{\mu}.
\]

(The exponential relation between the periodic principal eigenvalue
of the elliptic operator $\mathcal{A}_{\left(u_{1},u_{2}\right)}$
and that of the semiflow $Q_{t}^{u,lin}$ is classical and not detailed
here.)

On one hand, to apply \cite[Theorem 2.4]{Weinberger_200}, we have
to find $\delta\in\left(0,1\right)$ and $\eta_{+}>0$ such that,
for all $\left(v_{1},v_{2}\right)\in\left[\left(0,0\right),\left(\eta_{+},\eta_{+}\right)\right]$:
\[
Q_{t}\left[\left(v_{1},v_{2}\right)+\left(u_{1},\tilde{u}_{2}-u_{2}\right)\right]-\left(u_{1},\tilde{u}_{2}-u_{2}\right)\geq\left(1-\delta\right)Q_{t}^{u,lin}\left[\left(v_{1},v_{2}\right)\right],
\]
 that is such that:
\[
\delta Q_{t}^{u,lin}\left[\left(v_{1},v_{2}\right)\right]\geq Q_{t}^{u,lin}\left[\left(v_{1},v_{2}\right)\right]+\left(u_{1},\tilde{u}_{2}-u_{2}\right)-Q_{t}\left[\left(v_{1},v_{2}\right)+\left(u_{1},\tilde{u}_{2}-u_{2}\right)\right].
\]
On the other hand, by definition of $Q^{u,lin}$, for all $\varepsilon>0$,
we have the existence of $\eta_{\varepsilon}>0$ such that, if $\left(v_{1},v_{2}\right)\in\left[\left(0,0\right),\left(\eta_{\varepsilon},\eta_{\varepsilon}\right)\right]$:
\[
\left|Q_{t}^{u,lin}\left[\left(v_{1},v_{2}\right)\right]+\left(u_{1},\tilde{u}_{2}-u_{2}\right)-Q_{t}\left[\left(v_{1},v_{2}\right)+\left(u_{1},\tilde{u}_{2}-u_{2}\right)\right]\right|\leq\varepsilon\max\left(\max_{\overline{C}}v_{1},\max_{\overline{C}}v_{2}\right).
\]

Hence it would be sufficient to show, for all $\left(v_{1},v_{2}\right)\in\left[\left(0,0\right),\left(\eta_{\varepsilon},\eta_{\varepsilon}\right)\right]$,
the following inequality:
\[
\varepsilon\max\left(\max_{\overline{C}}v_{1},\max_{\overline{C}}v_{2}\right)\leq\delta\min\left(\min_{\overline{C}}Q_{t}^{u,lin}\left[\left(v_{1},v_{2}\right)\right]_{1},\min_{\overline{C}}Q_{t}^{u,lin}\left[\left(v_{1},v_{2}\right)\right]_{2}\right),
\]
 which is a straightforward consequence of the positivity of $\mathcal{A}_{\left(u_{1},u_{2}\right)}$
and of the instability of $\left(u_{1},u_{2}\right)$ (fixing for
instance $\delta=\frac{1}{2}$ and then choosing $\varepsilon$ small
enough). Finally we define $\eta_{+}=\eta_{\varepsilon}$.

Applying the same sketch of proof and being careful with the signs,
we prove the existence of $\eta_{-}>0$ such that, for all $\left(v_{1},v_{2}\right)\in\left[\left(0,0\right),\left(\eta_{-},\eta_{-}\right)\right]$:
\[
-Q_{t}\left[-\left(v_{1},v_{2}\right)+\left(u_{1},\tilde{u}_{2}-u_{2}\right)\right]+\left(u_{1},\tilde{u}_{2}-u_{2}\right)\geq\frac{1}{2}Q_{t}^{u,lin}\left[\left(v_{1},v_{2}\right)\right],
\]
whence a second inequality is established:
\[
c_{-}^{\star}\left(\left(u_{1},\tilde{u}_{2}-u_{2}\right),\left(0,0\right)\right)\geq\inf_{\mu>0}\frac{-\lambda_{1,per}\left(-\mu^{2}\text{diag}\left(1,d\right)-\mathcal{A}_{\left(u_{1},u_{2}\right)}\right)}{\mu}.
\]

It is worthy to point out that both spreading speeds are estimated
from below by the same quantity.

To conclude, we just have to notice the following inequality, true
for all $\mu>0$:
\[
\lambda_{1,per}\left(-\mu^{2}\text{diag}\left(1,d\right)-\mathcal{A}_{\left(u_{1},u_{2}\right)}\right)\leq-\mu^{2}\min\left(1,d\right)+\lambda_{1,per}\left(-\mathcal{A}_{\left(u_{1},u_{2}\right)}\right)<0.
\]

In particular, from:
\[
\frac{-\lambda_{1,per}\left(-\mu^{2}\text{diag}\left(1,d\right)-\mathcal{A}_{\left(u_{1},u_{2}\right)}\right)}{\mu}\geq\inf_{\mu>0}\left(\mu\min\left(1,d\right)-\frac{\lambda_{1,per}\left(-\mathcal{A}_{\left(u_{1},u_{2}\right)}\right)}{\mu}\right),
\]
we deduce the following estimate:
\[
\inf_{\mu>0}\frac{-\lambda_{1,per}\left(-\mu^{2}\text{diag}\left(1,d\right)-\mathcal{A}_{\left(u_{1},u_{2}\right)}\right)}{\mu}\geq2\sqrt{\min\left(1,d\right)\left|\lambda_{1,per}\left(-\mathcal{A}_{\left(u_{1},u_{2}\right)}\right)\right|}>0.
\]
\end{proof}

\subsection{Existence of pulsating fronts connecting both extinction states}

We are now able to state rigorously the existence of pulsating fronts
thanks to Fang\textendash Zhao \cite{Fang_Zhao_2011}. 
\begin{thm}
For any $k>k^{\star}$, there exists $c\in\mathbb{R}$ and $\left(\varphi_{1},\varphi_{2}\right)\in\mathcal{C}\left(\mathbb{R}^{2},\mathbb{R}^{2}\right)$
such that the following properties hold.
\begin{enumerate}
\item $\varphi_{1}$ and $\varphi_{2}$ are respectively non-increasing
and non-decreasing with respect to their first variable, generically
noted $\xi$. 
\item $\varphi_{1}$ and $\varphi_{2}$ are periodic with respect to their
second variable, generically noted $x$.
\item As $\xi\to+\infty$, 
\[
\max_{x\in\left[0,L\right]}\left|\left(\varphi_{1},\varphi_{2}\right)\left(-\xi,x\right)-\left(\tilde{u}_{1},0\right)\left(x\right)\right|+\max_{x\in\left[0,L\right]}\left|\left(\varphi_{1},\varphi_{2}\right)\left(\xi,x\right)-\left(0,\tilde{u}_{2}\right)\left(x\right)\right|\to0.
\]
\item $\left(u_{1},u_{2}\right):\left(t,x\right)\mapsto\left(\varphi_{1},\varphi_{2}\right)\left(x-ct,x\right)$
is a classical solution of $\left(\mathcal{P}\right)$.
\end{enumerate}
\end{thm}
\begin{rem*}
For any $\xi_{0}\in\mathbb{R}$, $\left(\xi,x\right)\mapsto\left(\varphi_{1},\varphi_{2}\right)\left(\xi+\xi_{0},x\right)$
is a pulsating front solution of $\left(\mathcal{P}\right)$ as well. 

Regarding the regularity of $\left(\varphi_{1},\varphi_{2}\right)$,
we recall that, even if Fang\textendash Zhao \cite{Fang_Zhao_2011}
(as well as Weinberger \cite{Weinberger_200}) worked in the framework
of continuous functions, by classical parabolic regularity, a continuous
solution of $\left(\mathcal{P}\right)$ is in $\mathcal{C}_{loc}^{1}\left(\mathbb{R},\mathcal{C}_{loc}^{2}\left(\mathbb{R},\mathbb{R}^{2}\right)\right)$.
Hence $\left(\varphi_{1},\varphi_{2}\right)$ is a fortiori in $\mathcal{C}_{loc}^{1}\left(\mathbb{R}^{2},\mathbb{R}^{2}\right)$.
This can be improved provided $f_{1}$ and $f_{2}$ are $\mathcal{C}^{1}$
with respect to $x$. Indeed, differentiating $\left(\mathcal{P}\right)$
with respect to $t$ and $x$ shows similarly that $\partial_{t}\left(u_{1},u_{2}\right)\in\mathcal{C}_{loc}^{1}\left(\mathbb{R},\mathcal{C}_{loc}^{2}\left(\mathbb{R},\mathbb{R}^{2}\right)\right)$
and $\partial_{x}\left(u_{1},u_{2}\right)\in\mathcal{C}_{loc}^{1}\left(\mathbb{R},\mathcal{C}_{loc}^{2}\left(\mathbb{R},\mathbb{R}^{2}\right)\right)$.
In such a case, $\left(\varphi_{1},\varphi_{2}\right)$ is at least
in $\mathcal{C}^{2}\left(\mathbb{R}^{2},\mathbb{R}^{2}\right)$.
\end{rem*}

\subsection*{Thanks}

The author would like to thank Grégoire Nadin for the attention he
paid to this work. The author is especially grateful to Jian Fang
who explained in person some technicalities with great patience and
clarity. The author would also like to thank the anonymous referee
for valuable input which considerably simplified the verification
of the counter-propagation hypothesis.

\bibliographystyle{amsplain}
\bibliography{ref}

\end{document}